\documentclass[12pt,leqno]{article}
\usepackage{amsfonts}
\usepackage{amsmath, amsthm, amsfonts, amssymb, color}
\usepackage{mathrsfs}
\usepackage{color}
\newtheorem{thm}{Theorem}[section]

\newtheorem{lem}[thm]{Lemma}

\theoremstyle{definition}

\newcommand{\scr}[1]{\mathscr #1}
\definecolor{wco}{rgb}{0.5,0.2,0.3}

\numberwithin{equation}{section} \theoremstyle{remark}
\newtheorem{rem}{Remark}[section]

\newcommand{\ua}{\uparrow}

\title{{\bf Functional Inequalities for Stable-Like Dirichlet Forms}\footnote{Supported in
 part by  Lab. Math. Com. Sys., NNSFC(11131003 and 11201073), SRFDP, the Fundamental Research Funds for the Central Universities and the Program for Excellent Young Talents and for New Century Excellent Talents in Universities
of Fujian (No.\ JA11051 and JA12053).} }
\author{
{\bf   Feng-Yu Wang$^{a), c)}$  and Jian Wang$^{b)}$
}
\\
\footnotesize{$^{a)}$School of Mathematical Sciences, Beijing Normal
University, Beijing 100875, China}\\
\footnotesize{$^{b)}$School of Mathematics and Computer Science, Fujian Normal University, Fuzhou 350007, China}\\
 \footnotesize{$^{c)}$Department of Mathematics,
Swansea University, Singleton Park, SA2 8PP, United Kingdom}\\ \footnotesize{wangfy@bnu.edu.cn, F.-Y.Wang@swansea.ac.uk, jianwang@fjnu.edu.cn}}
\begin{document}
\allowdisplaybreaks
\def\R{\mathbb R}  \def\ff{\frac} \def\ss{\sqrt} \def\B{\mathbf
B}
\def\N{\mathbb N} \def\kk{\kappa} \def\m{{\bf m}}
\def\dd{\delta} \def\DD{\Delta} \def\vv{\varepsilon} \def\rr{\rho}
\def\<{\langle} \def\>{\rangle} \def\GG{\Gamma} \def\gg{\gamma}
  \def\nn{\nabla} \def\pp{\partial} \def\E{\scr E}
\def\d{\text{\rm{d}}} \def\bb{\beta} \def\aa{\alpha} \def\D{\scr D}
  \def\si{\sigma} \def\ess{\text{\rm{ess}}}
\def\beg{\begin} \def\beq{\begin{equation}}  \def\F{\scr F}
\def\Ric{\text{\rm{Ric}}} \def\Hess{\text{\rm{Hess}}}
\def\e{\text{\rm{e}}} \def\ua{\underline a} \def\OO{\Omega}  \def\oo{\omega}
 \def\tt{\tilde} \def\Ric{\text{\rm{Ric}}}
\def\cut{\text{\rm{cut}}} \def\P{\mathbb P} \def\ifn{I_n(f^{\bigotimes n})}
\def\C{\scr C}      \def\aaa{\mathbf{r}}     \def\r{r}
\def\gap{\text{\rm{gap}}} \def\prr{\pi_{{\bf m},\varrho}}  \def\r{\mathbf r}
\def\Z{\mathbb Z} \def\vrr{\varrho} \def\ll{\lambda}
\def\L{\scr L}\def\Tt{\tt} \def\TT{\tt}\def\II{\mathbb I}
\def\i{{\rm in}}\def\Sect{{\rm Sect}}  \def\H{\mathbb H}
\def\M{\scr M}\def\Q{\mathbb Q} \def\texto{\text{o}} \def\LL{\Lambda}
\def\Rank{{\rm Rank}} \def\B{\scr B} \def\i{{\rm i}} \def\HR{\hat{\R}^d}
\def\to{\rightarrow}\def\l{\ell}
\def\8{\infty}\def\I{1}

\maketitle

\begin{abstract}  Let $V\in C^2(\R^d)$ such that $\mu_V(\d x):= \e^{-V(x)}\,\d x$ is a probability measure, and let $\aa\in (0,2)$. Explicit criteria are presented for the $\aa$-stable-like Dirichlet form
$$\E_{\aa,V}(f,f):= \iint_{\R^d\times\R^d} \ff{|f(x)-f(y)|^2}{|x-y|^{d+\alpha}}\,\d y\,\e^{-V(x)}\,\d x$$ to satisfy Poincar\'e-type (i.e.,  Poincar\'e, weak Poincar\'e and super Poincar\'e) inequalities. As applications, sharp functional inequalities are derived for the Dirichlet form with $V$ having some typical growths.
  Finally, the main  result  of \cite{MRS} on the Poincar\'e inequality is strengthened. \end{abstract} \noindent
 \vskip 2cm

\section{Introduction}

Functional inequalities are powerful and efficient tools to analyze
Markov semigroups and their generators, see e.g.\ \cite{WBook} for a
general theory of functional inequalities and applications. In
particular, the Nash/Sobolev inequalities are corresponding to
uniform heat kernel upper bounds of the semigroup, the log-Sobolev
inequality is equivalent to Nelson's hypercontractivity (\cite{N}) of
the semigroup, the super log-Sobolev inequality (also called the
log-Sobolev inequality with parameter) is equivalent to the
supercontractivity and in some cases implies the ultracontractivity
of the semigroup,  the Poincar\'e inequality is equivalent to the
exponential convergence of the semigroup, and the weak Poincar\'e
inequality characterizes various convergence rates of the semigroup
slower than exponential,  see e.g.\ \cite{Davies, Gross, DS, RW01,
W03} for details. As a general version of functional inequalities
stronger than the Poincar\'e one, the super Poincar\'e inequality is
equivalent to the uniform integrability   of the semigroup, and
also the absence of the essential spectrum of the generator if the
semigroup has an asymptotic density, see \cite{W00a, W00b, GW, W04}
for details.

To establish functional inequalities, many explicit criteria have
been proved for diffusion processes and Markov chains, but rare is
known for L\'evy type jump processes. Of course, using subordination
techniques, functional inequalities  for a class of jump processes
can be deduced from known ones of diffusion processes, see \cite{BM,
W07, SW11, GM} and \cite[Chapter 12.3]{SSV} (in an abstract setting)
for details.  However, in general it is difficult (and impossible in
many cases) to identify a L\'evy type jump process as subordination
of a diffusion process. So, it is necessary to provide general
criteria to verify functional inequalities for L\'evy type jump
processes. We remark that using harmonic analysis technique, a
sufficient condition for the Poincar\'e inequality to hold,  see
(\ref{EF}) below,  has been presented in \cite{MRS}. As pointed out
after Corollary \ref{C1.5} below,  this condition excludes many
typical examples which possess the even stronger super Poincar\'e
inequality.   The purpose of this paper is to find out sharp and
easy to check sufficient conditions for general functional
inequalities of stable-like jump processes.

To make the paper easy to follow, let us start with a simple example,
 i.e. the Ornstein-Uhlenbeck process driven by the $\aa$-stable process.
 Let $\DD$ be the Laplacian on $\R^d$. Consider the Ornstein-Uhlenbeck operator
$$A_\aa f(x):= -(-\DD)^{\aa/2}f(x) -\<x, \nn f (x)\> ,\ \ \ f\in C_0^\infty(\R^d)$$
for $\aa\in (0,2).$ Then the associated Markov semigroup has a
unique invariant (but not reversible, see \cite{ABW}) probability
measure $\mu_\aa$, which is identified by the Fourier transformation
$$\hat\mu_\aa(\xi):=  \int_{\R^d} \e^{i \<x,\xi\>}\,\mu_\alpha(\d x)
=\e^{-\frac{1}{\alpha}|\xi|^\alpha},\quad \xi\in\R^d.
$$
For any $f\in C_0^\infty(\R^d)$, the set of all smooth functions on
$\R^d$ with compact support, we have ({see \cite[Proposition
4.1]{LR} or \cite[(1.9)]{RW}}) \beq\label{DA}\E_\aa(f,f):=
-\int_{\R^d} f A_\aa  f\d\mu_\aa =\ff 1 2 \iint_{\R^d\times \R^d}
\ff{|f(x)-f(y)|^2}{|x-y|^{d+\aa}}\,\d y\mu_\aa(\d
x).\end{equation}Let $\D(\E_\aa)=\{f\in L^2(\mu_\aa):\
\E_\aa(f,f)<\infty\}.$ According to \cite[Example 3.2(2)]{RW}, the
semigroup $P_t^\aa$ generated by $A_\aa$ is not hyperbounded, i.e.
$\|P_t^\aa\|_{L^p(\mu_\aa)\to L^q(\mu_\aa)}=\infty$ for any $t\ge 0$
and $q>p\ge 1.$ Therefore, the log-Sobolev inequality of $\E_\aa$
does not hold. In fact, since
\beq\label{AP}\frac{1}{c(1+|x|^2)^{(d+\alpha)/2}}\,\d x\le
\mu_\alpha(\d x)\le \frac{c}{(1+|x|^2)^{(d+\alpha)/2}}\,\d
x\end{equation}  holds for some constant $c>1$,  see e.g.
\cite[Theorem 2.1]{BG} or \cite[(1.5)]{CZQ},  Corollary
\ref{C1.2}(2) below provides  a stronger statement, i.e. the super
Poincar\'e inequality is  not available neither. Recall that the
log-Sobolev inequality
$$\mu_\aa(f^2\log f^2) \le C \E_\aa(f,f),\ \ f\in \D(\E_\aa), \mu_\aa(f^2)=1$$
holds for some constant $C>0$ if and only if the super Poincar\'e
inequality
$$\mu_\aa(f^2)\le r\E_\aa(f,f) +\exp\Big(c\big(1+r^{-1}\big)\Big)\mu_\aa(|f|)^2,\ \ r>0, f\in \D(\E_\aa)$$
holds for some constant $c>0$. On the other hand,  Corollary \ref{C1.2}(1) implies that the Poincar\'e inequality
$$\mu_\aa(f^2)\le C\E_\aa(f,f),\ \ f\in \D(\E_\aa), \mu_\aa(f)=0$$ holds for some constant $C>0$, which has been open for a long time. Therefore, for this typical example, the best possibility among functional inequalities mentioned above is the Poincar\'e inequality.

Now, as a generalization of (\ref{DA}), we consider
\beg{equation*}\beg{split} &\E_{\aa,V}(f,g):= \iint_{\R^d\times\R^d} \ff{(f(x)-f(y))(g(x)-g(y))}{|x-y|^{d+\aa}}\,\d y\mu_V(\d x), \\
 &\D(\E_{\aa,V}):= \Big\{f\in L^2(\mu_V):\ \E_{\aa, V}(f,f)<\infty\Big\},\end{split}\end{equation*}  where $V$ is a measurable function on $\R^d$ such that
 $$\mu_V(\d x):= \ff 1 {\int_{\R^d} \e^{-V(x)}\d x} \,\e^{-V(x)}\d x$$ is a probability measure.
 Then $(\E_{\aa,V}, \D(\E_{\aa, V}))$ is a symmetric Dirichlet form on $L^2(\mu_V)$.  Let $P_t^{\aa,V}$ be the associated  Markov semigroup.
 Let
 \beg{equation*}\beg{split} &h(r)= \inf_{|x|\le r} \e^{V(x)},\ \ H(r)= \sup_{|x|\le r} \e^{V(x)},\\
 & \Phi(r)= \inf_{|x|\ge r} \ff{\e^{V(x)}}{(1+|x|)^{d+\aa}},\ \ \Phi^{-1}(r)=\inf\big\{s\ge 0:\ \Phi(s)\ge r\big\},\ \ r>0,\end{split}\end{equation*} where we set $\inf\emptyset =\infty$ by convention. Moreover, let
\beg{equation*}\beg{split} \Psi_1(r)& = \bigg(\sup_{|x|\le r}
\frac{(1+|x|)^{d+\aa}}{\e^{V(x)}} \bigg) \sup_{x\in\R^d}
\ff{\e^{V(x)}}{(1+|x|)^{d+\aa}},\\
\Psi_2(r)&= \ff 1 {\mu_V(B(0,r))^2} \sup_{x\in B(0,r)}\int_{B(0,r)}|y-x|^{d+\aa}\e^{-2V(y)}\d y, \ \ r>0.
\end{split}\end{equation*}
The main result of the paper is the following

\beg{thm}\label{T1.1}  Let $\int_{\R^d}\e^{-V(x)}\d x<\infty$ such that $\mu_V$ is a well defined probability measure.
 \beg{enumerate}
\item[$(1)$] If $\e^{-V}\in C_b^2(\R^d)$ such that
\beq\label{A1} \limsup_{r\to\infty} \Big\{r^{d+\aa-1} \sup_{|x|\ge r-1} |\nn \e^{-V}(x)| + r^{d+\aa-2} \sup_{|x|\ge r-1} \e^{-V(x)}\Big\}=0\end{equation} and $ \Phi(0) >0,$ then the Poincar\'e inequality
\beq\label{P} \mu_V(f^2)\le C \E_{\aa,V}(f,f),\ \ f\in \D(\E_{\aa,V}), \mu_V(f)=0\end{equation} holds for some constant $C>0$.
\item[$(2)$] If  $\e^{-V}\in C_b^2(\R^d)$ such that \eqref{A1} holds and $\Phi(r)\uparrow\infty$ as $r\uparrow\infty,$ then there exist constants $c_1,c_2>0$ such that  the super Poincar\'e inequality
\beq\label{SP} \mu_V(f^2)\le r \E_{\aa,V}(f,f) +\bb(r)\mu_V(|f|)^2,\ \ r>0, f\in \D(\E_{\aa,V})\end{equation} holds for
$$
\beta(r)=c_1\Big(1+ r^{-d/\alpha}\big\{h\circ\Phi^{-1}( c_2 r^{-1})  \big\}^{-1-d/\alpha}
\big\{H\circ\Phi^{-1}(c_2r^{-1}) \big\}^{2+d/\alpha}\Big),\ \ r>0.
$$
\item[$(3)$] There exists  a universal constant $c >0$ such that  the weak Poincar\'e inequality
\beq\label{WP} \mu_V(f^2)\le \tt\bb(r) \E_{\aa,V}(f,f) +r
\|f\|_\infty^2,\ \ r>0, f\in\D(\E_{\aa,V}), \mu_V(f)=0\end{equation}
holds for
$$\tt\bb(r):= \inf \Big\{ \big(c\Psi_1(R)\big)\wedge \Psi_2(R):\  \mu_V(B(0,R)^c)\le \ff r {1+r}\Big\}<\infty,\ \ r>0.$$
\end{enumerate}\end{thm}

Although we assume in Theorem \ref{T1.1}(1)-(2) that $\e^{-V}$ is at
least $C^2$-smooth, the assertions work also for singular case by
using perturbation results of functional inequalities, see
\cite{CWW}. To illustrate this result, below we consider some
typical families of $V$ with different type growths:  for faster
growth of $V$ one derives stronger functional inequality. When we
apply Theorem \ref{T1.1}(3) to derive weak Poincar\'e inequalities
for these families of $V$, the function $\Psi_1$ in the definition
of $\tt\bb$ is better than $\Psi_2$. On the other hand, however,
$\Psi_2$ is always finite but in some cases $\Psi_1$ is infinite.
So, in general these two functions are not comparable.

According to (\ref{AP}), in the following result $\mu_V$ is a natural extension to $\mu_\aa$, i.e. when $\vv=\aa$ a Poincar\'e type inequality for $\E_{\aa,V}$ and $\mu_V$ is equivalent to that for $\E_\aa$ and $\mu_\aa$. In particular, as mentioned above, this result  implies that $\E_\aa$ satisfies the Poincar\'e inequality but not the super Poincar\'e inequality.

\beg{cor}\label{C1.2} Let $V(x)= \ff 1 2 (d+\vv)\log (1+|x|^2),\   \vv>0.$
\beg{enumerate} \item[$(1)$] The Poincar\'e inequality $(\ref{P})$ holds for some   constant $C>0$ if and only if $\vv\ge\aa$.
\item[$(2)$] The super Poincar\'e inequality $(\ref{SP})$ holds for some function $\bb: (0,\infty)\to (0,\infty)$ if and only if $\vv>\aa$, and in this case there exists a constant $c>0$ such that the inequality holds with
$$\bb(r)= c\Big(1+ r^{-\ff d\aa-\ff{(d+\vv)(2\aa+d)}{\aa(\vv-\aa)}}\Big),\ \ \ r>0,$$   and equivalently,
$$\|P_t^{\aa,V}\|_{L^1(\mu_V)\to L^\infty(\mu_V)}\le \ll \Big(1+ t^{-\ff d\aa-\ff{(d+\vv)(2\aa+d)}{\aa(\vv-\aa)}}\Big),\ \ \ r>0$$ holds for some constant $\ll>0.$
\item[$(3)$] If $\vv\in (0,\aa)$, then there exists a constant $c>0$ such that the weak Poincar\'e inequality $(\ref{WP})$ holds for
$$\tt\bb(r)= c\Big(1+ r^{-(\aa-\vv)/\vv}\Big),\ \ r>0.$$ Consequently, there exists a constant $\ll>0$ such that
$$\|P_t^{\aa,V}-\mu_V\|_{L^\infty(\mu_V)\to L^2(\mu_V)}^2\le \ff \ll {t^{\vv/(\aa-\vv)}},\ \ t>0. $$
 This $\tt\bb$ is sharp in the sense that $(\ref{WP})$ does not hold if $\lim_{r\to 0} r^{(\aa-\vv)/\vv}\tt\bb(r)=0.$  \end{enumerate}\end{cor}

Since $\vv=\aa$ in Corollary \ref{C1.2}  is the critical situation for the Poincar\'e inequality, we consider below lower order perturbations of the corresponding $V$.

 \beg{cor}\label{C1.3} Let $V(x)= \ff 1 2 (d+\aa)\log (1+|x|^2) + \vv\log\log(\e+|x|^2),\  \vv\in\R.$
\beg{enumerate}
\item[$(1)$] The super Poincar\'e inequality $(\ref{SP})$ holds for some $\bb$ if and only if $\vv>0$, and in this case it holds with  $$\bb(r)= \exp\Big[c\Big(1+r^{-1/\vv}\Big)\Big]$$ for some constant $c>0$, so that when $\vv>1$,
$$\|P_t^{\aa,V}\|_{L^1(\mu_V)\to L^\infty(\mu_V)}\le \exp\Big[\ll \Big(1+ t^{-1/(\vv-1)}\Big)\Big],\ \ \ t>0$$ holds for some constant $\ll>0.$
\item[$(2)$] The super Poincar\'e inequality in $(1)$ is sharp in the sense that $(\ref{SP})$ does not hold if
$$\lim_{r\to 0}r^{1/\vv}\log\bb(r)=0.$$
\item[$(3)$] The log-Sobolev inequality
\beq\label{LS} \mu_V(f^2\log f^2)\le C\E_{\aa,V}(f,f),\ \ f\in \D(\E_{\aa,V}), \mu_V(f^2)=1\end{equation} holds for some constant $C>0$ if and only if $\vv\ge 1.$
\item[$(4)$] The Poincar\'e inequality $(\ref{P})$  holds for some constant $C>0$ if and only if $\vv\ge 0$, and there exists a universal constant $c>0$ such that for $\vv<0$ the weak Poincar\'e inequality
$(\ref{WP})$ holds with  $$\widetilde{\bb}(r)= c \Big(1+ \log^{-\vv}
\big(1+r^{-1}\big)\Big),\ \ r>0.$$
Consequently, for $\vv<0$ there exist   constants $\ll_1,\ll_2>0$
such that
 $$\|P_t^{\aa,V}-\mu_V\|_{L^\infty(\mu_V)\to L^2(\mu_V)}\le \exp\Big[\ll _1-\ll_2t^{1/(1-\varepsilon)}\Big],\ \ \ t>0.$$
This $\tt\bb$ is sharp in the sense that for $\vv<0$ the weak Poincar\'e inequality $(\ref{WP})$ does not hold if $\lim_{r\to 0} \tt\bb(r)\log^\vv (1+r^{-1})   =0.$ \end{enumerate}
\end{cor}

Below we consider a family of $V$ with slower growth such that $\mu_V$ is a probability measure, for which merely the weak Poincar\'e inequality is available.

  \beg{cor}\label{C1.4} Let $V(x)= \ff d 2 \log (1+|x|^2) + \vv\log\log(\e+|x|^2),\ \ \vv>1.$
Then there exist some constants $c_1,c_2>0$ such that the weak Poincar\'e inequality $(\ref{WP})$ holds with  $$\widetilde{\bb}(r)= c_1\exp\Big[c_2r^{-1/(\vv-1)}\Big].$$ Consequently, there exists some constant $\ll>0$ such that
 $$\|P_t^{\aa,V}-\mu_V\|_{L^\infty(\mu_V)\to L^2(\mu_V)}\le \ll \Big[\log(1+t)\Big]^{1-\vv},\ \ \ t>0.$$
 This $\tt \bb$ is sharp in the sense that the weak Poincar\'e inequality $(\ref{WP})$ does not hold if $\lim_{r\to 0} r^{1/(\vv-1)} \log\tt \bb(r)=0.$
\end{cor}

Finally, we consider two families of  $V$ with stronger growths than
all those presented above, so that the rather stronger super
Poincar\'e inequality is available.

\beg{cor} \label{C1.5} $(1)$  Let  $V(x)=\log^{1+\vv}(1+|x|^2),\ \vv>0.$ Then there exists a constant $c>0$ such that $(\ref{SP})$ holds for
$$\bb(r)= c+c r^{-2(\aa+d)/\aa} \exp\big[c\log ^{1/(1+\vv)} (1+r^{-1})\big],  \ \ r>0.$$ Consequently, there exists a constant $\ll>0$ such that
$$\|P_t^{\aa,V}\|_{L^1(\mu_V)\to L^\infty(\mu_V)}\le \ll + \ll t^{-2(\aa+d)/\aa} \exp\big[\ll\log^{1/(1+\vv)} (1+t^{-1})\big],\ \ t>0.$$

$(2)$ Let $V(x)=   (1+|x|^2)^\vv,\ \ \vv>0.$  Then there exists a constant $c>0$ such that the super Poincar\'e inequality $(\ref{SP})$ holds for  $$\bb(r)=  c\Big(1 + r^{-2(\aa+d)/\aa}\log^{(2\aa+d)(d+\aa)/(2\vv \aa)}(1+r^{-1})\Big),\ \ \ r>0,$$ and consequently,
 $$\|P_t^{\aa,V}\|_{L^1(\mu_V)\to L^\infty(\mu_V)}\le\ll\Big(1  + t^{-2(\aa+d)/\aa}\log^{(2\aa+d)(d+\aa)/(2\vv \aa)}(1+t^{-1})\Big),\ \ \ t>0$$ holds for some constant $\ll>0$.
\end{cor}

\
We remark that the following sufficient condition for $\E_{\aa,V}$ to satisfy the Poincar\'e inequality has been presented in \cite{MRS}: $V\in C^2(\R^d)$ such that \beq\label{EF} \lim_{|x|\to\infty} \big\{\dd |\nn V|^2-\DD V\} =\infty\ \ \text{for\ some\ constant}\ \dd\in (0,1/2).\end{equation} Obviously, this condition does not hold for $V$ in Corollaries \ref{C1.2}-\ref{C1.5}(1). In the situation of Corollary \ref{C1.5}(2), (\ref{EF}) holds if and only if $\vv>\ff 1 2$. In this case, using the argument of \cite{MRS}, we are able to confirm
the super Poincar\'e inequality for (see Theorem \ref{TA} below)
$$\bb(r) = \exp\Big[c\Big(1+r^{-2\vv/(\alpha(2\vv-1))}\Big)\Big],\ \ r>0$$ for some constant $c>0$, which is however much worse than the one given in Corollary \ref{C1.4}(2).
We also mention that sufficient conditions for a (non-symmetric)
$L^2$-generator of L\'{e}vy driven Ornstein-Uhlenbeck processes to
satisfy Poincar\'{e} inequality have been investigated in
\cite[Section 5]{K}, where the proof is based on exact asymptotics
for a distribution density of certain L\'{e}vy functionals; however,
extensions to the present setting are not yet available.

\
The proof of Theorem \ref{T1.1} is based on Lyapunov type conditions considered in \cite{CGWW}. To verify these conditions, we first characterize in Section 2 the infinitesimal generator of $(\E_{\aa,V}, \D(\E_{\aa,V}))$, then present complete proofs of the above results in Section 3 and Section 4. Finally, in Section 5 we present a result on the super Poincar\'e inequality using a weaker version of condition (\ref{EF}) by allowing $\dd$ to approach $1$, such that the main result in \cite{MRS} on the Poincar\'e inequality is strengthened.

\section{The infinitesimal generator of $\E_{\aa,V}$}

We first introduce some facts concerning the Dirichlet form and generator of the $\aa$-stable process. Let
$$ \C_\aa =\big\{f\in C^2(\R^d):\ \|\nn f\|_\infty<\infty, |f|\le C(1+|\cdot|^r)\ \text{holds\ for \ some \ } C>0, r\in (0,\aa)\big\}.$$
For any $f\in \C_\aa$, there exist    constants $C>0$ and $r\in (0,\aa)$ such that
\beg{equation*}\beg{split}|f(x+z)- f(x)-  & \<\nn f ,z\>1_{\{|z|\le 1\}}|\frac{1}{|z|^{d+\alpha}}\\
&\le \ff{ \sup_{B(x,1)}\|\nn^2 f\|} {|z|^{d+\aa-2}}  1_{\{|z|\le
1\}} +
\ff{C(1+|x|^r+|z|^r)}{|z|^{d+\aa}}1_{\{|z|>1\}}.\end{split}\end{equation*}
Then for $f\in \C_\aa$, \beq\label{DAD}-(-\DD)^{\aa/2} f := C_\aa
\int_{\R^d} \Big(f(\cdot+z)-f -\<\nn f ,z\>1_{\{|z|\le 1\}}\Big)
\ff{\d z}{|z|^{d+\aa}} \end{equation} is a well-defined    locally
bounded measurable function, where $C_\aa>0$ is a constant such that
({see \cite[Example 32.7]{Sa}}), \beq\label{II} \ff {2}
{C_\aa}\int_{\R^d} \Big(f\,(-\DD)^{\aa/2} g\Big)(x) \,\d
x=\E_\aa^{(0)}(f,g),\ \ f,g\in C_0^2(\R^d),\end{equation}  where
$$\E_\aa^{(0)}(f,g):= \iint_{\R^d\times\R^d} \ff{(f(x)-f(y))(g(x)-g(y))}{|x-y|^{d+\aa}}\,\d y\,\d x.$$

Next, for $f\in C_0^2(\R^d)$ and $g\in\C_\aa$,  there exist constants $C,R>0$ and $r\in (0,\aa)$ such that
\beg{equation*}\beg{split} |f(x)-f(y)|&\cdot |g(x)-g(y)|\\
&\le C |x-y|^21_{\{|x-y|\le R\}}+ C(|x|^r+|y|^r+1)  1_{(\text{supp}f \times  \text{supp}f)^c}(x,y)1_{\{|x-y|>R\}},\end{split}\end{equation*}   so that
$\E_{\aa}^{(0)}(f,g)\in\R$ is  well-defined.

Moreover, since for any function $g\in \C_\aa$  there exist $\{g_n\}_{n\ge 1}\subset C_0^2(\R^d)$ such that  $ \|\nn g_n\|_\infty \le C, |g_n|\le C(1+|\cdot|^r)$ holds for some constants $C>0$ and $r\in (0,\aa)$, and that  $g_n\to g, \nn g_n\to \nn g$ and $\nn^2 g_n\to \nn^2g$ uniformly on compact sets, (\ref{II}) implies that
\beq\label{INT} \ff {2} {C_\aa}\int_{\R^d} \Big(f\,(-\DD)^{\aa/2} g\Big)(x) \,\d x= \E_\aa^{(0)}(f,g),\ \ f\in C_0^2(\R^d), g\in\C_\aa.\end{equation}

Finally, if $\e^{-V}\in C_b^1(\R^d)$ and $g\in \C_\aa$, then
$$|g(x)-g(y)|\cdot |\e^{-V(x)}-\e^{-V(y)}|\le C |x-y|^21_{\{|x-y|\le 1\}} + \ff{C (1+|x|^r+|y|^r)}{|x-y|^{d+\aa}}1_{\{|x-y|>1\}}$$ holds for some constant $C>0$ and $r\in (0,\aa).$ Therefore, in conclusion,    if $\e^{-V}\in C_b^2(\R^d)$ and $g,    g\e^{-V}\in \C_\aa$, then
\beq\label{LL}\beg{split} L_{\aa,V} g :=&\ff 2 {C_\aa} \Big( g\e^V(-\DD)^{\aa/2}\e^{-V} -\e^V(-\DD)^{\aa/2}(\e^{-V}g)\Big)\\
&-\e^V\int_{\R^d}\ff{(g-g(y))(\e^{-V}-\e^{-V(y)})}{|\cdot-y|^{d+\aa}}\,\d y\end{split}\end{equation} gives rise to a locally bounded measurable function.

\beg{prp}\label{P2.1} Assume that $\e^{-V}\in C_b^2(\R^d)$. For any $f\in C_0^2(\R^d)$ and $g\in \C_\aa$ such that $\e^{-V}g\in\C_\aa$,
$$\E_{\aa,V}(f,g)= -\int_{\R^d}fL_{\aa,V} g \,\d\mu_V.$$
\end{prp}

\beg{proof}  Since $f\e^V, fg\e^V\in C_0^2(\R^d)$ and $\e^{-V}g, \e^{-V}\in\C_\aa$, it follows from (\ref{INT}) that
\begin{align*} -\int_{\R^d} fL_{\aa, V}g\, \d\mu_V
 =& \frac{2}{C_\aa}\int f\e^{V}(-\Delta)^{\alpha/2}(\e^{-V}g)\,\d\mu_V-\ff {2}{C_\aa}\int fg\e^{V}(-\Delta)^{\alpha/2}\e^{-V}\,\d\mu_V\\
&
+\iint_{\R^d\times\R^d} \ff{f(x)(g(x)-g(y))(\e^{-V(x)}-\e^{-V(y)})}{|x-y|^{d+\aa}}\,\d y\,\d x\\
=&\,\, \E_\aa^{(0)} (\e^{-V}g, f) -\E_\aa^{(0)} (fg, \e^{-V})\\
&+\iint_{\R^d\times\R^d} \ff{f(x)(g(x)-g(y))(\e^{-V(x)}-\e^{-V(y)})}{|x-y|^{d+\aa}}\,\d y\,\d x\\
=&\iint_{\R^d\times\R^d} \ff 1 {|x-y|^{d+\aa}} \times\Big\{ \big(\e^{-V(x)}g(x)- \e^{-V(y)}g(y)\big)\big(f(x)-f(y)\big) \\
&\qquad\qquad\qquad\qquad\qquad-\big((fg)(x)-(fg)(y)\big)
\big(\e^{-V(x)}- \e^{-V(y)}\big) \\
&\qquad\qquad\qquad\qquad\qquad+f(x)\big(g(x)-g(y)\big)\big(\e^{-V(x)}-\e^{-V(y)}\big)\Big\}\, \d y\,\d x\\
=&\iint_{\R^d\times \R^d}
\ff{(f(x)-f(y))(g(x)-g(y))}{|x-y|^{d+\aa}}\d y\, \e^{-V(x)}\,\d x\\
 =&\,\,\E_{\aa,V}(f,g).\end{align*} \end{proof}

According to Proposition \ref{P2.1}, the operator $(L_{\aa,V},
C_0^2(\R^d))$ is symmetric on $L^2(\mu_V)$; on the other hand,
  $(\E_{\aa,V}, C_0^\infty(\R^d))$ is closable and it is easy to see that its closure coincides with $(\E_{\aa,V},\D(\E_{\aa,V}))$. Moreover,
 combining (\ref{DAD}) with (\ref{LL}),  we obtain the following result with explicit expression of $L_{\aa,V}$.

 \beg{prp}\label{P2.2}  Assume that $\e^{-V}\in C_b^2(\R^d)$. For any $f\in \C_\aa$ such that $f\e^{-V}\in\C_\aa$, \begin{equation*}\aligned
L_{\aa,V} f(x)
=& \int_{\R^d} \Big(f(x+z)-f(x)-\<\nn f(x),z\>1_{\{|z|\le 1\}}\Big) \ff{1+\e^{V(x)-V(x+z)}}{|z|^{d+\aa}}\,\d z\\
&\quad+ \int_{\{|z|\le 1\}}
\<\nn f(x), z\>\Big(\e^{V(x)-V(x+z)}-1\Big)\frac{\d z}{|z|^{d+\alpha}}.
\endaligned \end{equation*} \end{prp}
\beg{proof} By (\ref{DAD}) we have
\begin{align*}L_{\aa,V,1}f(x):=& \ff {2}{C_\aa} \Big(f(x)\e^{V(x)} (-\DD)^{\aa/2}\e^{-V}(x)-\e^{V(x)}(-\DD)^{\aa/2}(\e^{-V}f)(x)\Big)\\
=& 2\int_{\R^d} \Big(f(x+z)-f(x)-\<\nn f(x),z\>1_{\{|z|\le 1\}}\Big) \ff{\e^{V(x)-V(x+z)}}{|z|^{d+\aa}}\,\d z\\
&\quad+ 2\int_{\{|z|\le 1\}}
\<\nn f(x), z\>\Big(\e^{V(x)-V(x+z)}-1\Big)\frac{\d z}{|z|^{d+\alpha}}.\end{align*}
On the other hand,
\begin{align*}L_{\aa,V,2}f(x):=&\e^{V(x)}\int_{\R^d} \ff{(f(x)-f(y))(\e^{-V(x)}-\e^{-V(y)})}{|x-y|^{d+\aa}}\,\d y\\
=&\int_{\R^d}\Big(f(y)-f(x)\Big)\frac{e^{-V(y)+V(x)}-1}{|y-x|^{d+\alpha}}\,\d y\\
=&\lim_{\varepsilon\to0}\Bigg[\int_{|z|\ge\varepsilon}\Big(f(x+z)-f(x)-\langle\nabla f(x), z\rangle1_{\{|z|\le 1\}}\Big)\frac{\e^{-V(x+z)+V(x)}-1}{|z|^{d+\alpha}}\,\d z\\
&\qquad+\int_{|z|\ge\varepsilon}\langle \nabla f(x), z\rangle1_{\{|z|\le 1\}}\frac{\e^{-V(x+z)+V(x)}-1}{|z|^{d+\alpha}}\,\d z\bigg]\\
=&\int_{\R^d}\Big(f(x+z)-f(x)-\langle\nabla f(x), z\rangle1_{\{|z|\le 1\}}\Big)\frac{\e^{-V(x+z)+V(x)}-1}{|z|^{d+\alpha}}\,\d z\\
&\qquad+\int_{\{|z|\le 1\}}\langle \nabla f(x),
z\rangle\Big(\e^{-V(x+z)+V(x)}-1\Big)\frac{1}{|z|^{d+\alpha}}\,\d
z.\end{align*} Combining both equalities above with \eqref{LL}, we
prove the desired assertion.
\end{proof}

Finally, the following result confirms the Lyapunov condition used in \cite{CGWW} for the study of super Poincar\'e inequalities.

\beg{prp}\label{P2.3} Assume   $\e^{-V}\in C_b^2(\R^d)$ and that $(\ref{A1})$ holds.  Let $\aa_0\in (0,1\land\aa)$ and let $\phi\in C^\infty(\R^d)$ such that  $\phi(x)=1+|x|^{\aa_0}$ for $|x|\ge 1$.  Then $\e^{-V},\phi, \phi\e^{-V}\in\C_\aa$. If moreover
$\Phi(0)>0$, then there exist constants $r_0,C_1,C_2>0$ such that
$$L_{\aa,V}\phi (x) \le -C_1 \Phi(|x|) \phi(x) +C_21_{\{|x|\le r_0\}},\ \ \  x\in\R^d.$$\end{prp}

\begin{proof} By (\ref{A1}) and the choice of $\phi$, it is easy to see that $\phi,  \e^{-V}\phi\in \C_\aa.$  Since $L_{\aa,V}\phi$ is locally bounded,  we only need to verify the conclusion for $|x|$ large enough.

Using the facts that $2\<x, z\>=|x+z|^2-|x|^2-|z|^2$ for all $x$, $z\in\R^d$, and $b^{\alpha_0}- a^{\alpha_0}\le \alpha_0a^{\alpha_0-1}(b-a)$ for any $a$, $b\ge0$, we get that for  $|x|$ large enough,
\begin{align*}&\int_{\{|z|\le 1\}}\Big(\phi(x+z)-\phi(x)-\<\nabla \phi(x), z\>\Big)\,\frac{\d z}{|z|^{d+\alpha}} \\
&\le\alpha_0|x|^{\alpha_0-1}\int_{\{|z|\le 1\}}\Big(|x+z|-|x|-\frac{\<x, z\>}{|x|}\Big)\,\frac{\d z}{|z|^{d+\alpha}} \\
&=\frac{1}{2}\alpha_0|x|^{\alpha_0-2}\int_{\{|z|\le 1\}}\Big(2|x+z|\cdot |x|-2|x|^2-|x+z|^2+|x|^2+|z|^2\Big)\,\frac{\d z}{|z|^{d+\alpha}} \\
&=\frac{1}{2}\alpha_0|x|^{\alpha_0-2}\int_{\{|z|\le 1\}}\Big(|z|^2-(|x|-|x+z|)^2\Big)\,\frac{\d z}{|z|^{d+\alpha}} \\
&\le\frac{1}{2}\alpha_0|x|^{\alpha_0-2}\int_{\{|z|\le 1\}}\,\frac{\d z}{|z|^{d+\alpha-2}}\\
& \le 1.\end{align*} Let $c_1= \sup_{|z|\le 1} \phi(z)$. Then  $\phi(x)\le c_1+1+|x|^{\alpha_0}$ holds for all $x\in \R^d$. Combining this  with
 $\phi(x)=1+|x|^{\aa_0}$ for $|x|\ge 1$, and the triangle inequality $(a+b)^{\alpha_0}\le a^{\alpha_0}+ b^{\alpha_0}$ for  $a$, $b\ge0$, we obtain that for $|x|$ large enough
\begin{align*}&\int_{\{|z|> 1\}}\Big(\phi(x+z)-\phi(x)\Big)\,\frac{1}{|z|^{d+\alpha}}\,\d z\\
&\le\int_{\{|z|> 1\}}\Big(c_1+|x+z|^{\alpha_0}-|x|^{\alpha_0}\Big)\,\frac{1}{|z|^{d+\alpha}}\,\d z\\
&\le \int_{\{|z|> 1\}}\Big(c_1+|z|^{\alpha_0}\Big)\,\frac{\d z}{|z|^{d+\alpha}} \\
&=:c_2<\infty.
  \end{align*} Therefore, for   $|x|$ large enough,
\beq\label{AA1}  \int_{\R^d}\Big(\phi(x+z)-\phi(x)-\nabla \phi(x)\cdot z\I_{\{|z|\le 1\}}\Big)\,\frac{\d z}{|z|^{d+\alpha}} \le 1+c_2.\end{equation}

Next, since $|x+z|^{\aa_0}-|x|^{\aa_0}\le |z|^{\aa_0}$, and for large enough $|x|,$
\beg{equation*}\beg{split} 1_{\{|x+z|\le |x|\}}\Big(|x+z|^{\alpha_0}-|x|^{\alpha_0}\Big)
&\le 1_{\{|x+z|\le 1\}}\Big(|x+z|^{\alpha_0}-|x|^{\alpha_0}\Big),\\
\sup_{|z|\ge |x|} \e^{-V(z)}&\le \ff {1} {\Phi(0)(1+|x|)^{d+\aa}}, \end{split}\end{equation*}
there exists a constant $c_3>0$ such that for $|x|$ large enough,
\begin{align*}& \int_{\{|z|> 1\}}\Big(\phi(x+z)-\phi(x)\Big)\,\frac{  \e^{V(x)-V(x+z)}}{|z|^{d+\alpha}}\,\d z \\
& \le \e^{V(x)}\int_{\{|z|> 1\}}\Big(c_1+|x+z|^{\alpha_0}-|x|^{\alpha_0}\Big)\,\frac{  \e^{-V(x+z)}}{|z|^{d+\alpha}}\,\d z\\
&\le  \int_{\{|z|> 1, |x+z|\le |x|\}}\Big(c_1+|x+z|^{\alpha_0}-|x|^{\alpha_0}\Big)\,\frac{  \e^{V(x)-V(x+z)}}{|z|^{d+\alpha}}\,\d z\\
&\qquad
 +  \int_{\{|z|> 1, |x+z|> |x|\}}\Big(c_1+|z|^{\alpha_0}\Big)\frac{  \e^{V(x)-V(x+z)}}{|z|^{d+\alpha}}\,\d z\\
&\le  \int_{\{|z|> 1, |x+z|\le 1\}}\Big(c_1+|x+z|^{\alpha_0}-|x|^{\alpha_0}\Big)\,\frac{  \e^{V(x)-V(x+z)}}{|z|^{d+\alpha}}\,\d z\\
&\qquad
 +  \int_{\{|z|> 1, |x+z|> |x|\}}\Big(c_1+|z|^{\alpha_0}\Big)\frac{  \e^{V(x)-V(x+z)}}{|z|^{d+\alpha}}\,\d z\\
&\le \e^{V(x)}\bigg(\inf_{|z|\le 1}\e^{-V(z)}\bigg)\int_{\{|z|> 1, |x+z|\le 1\}}\Big(c_1+1-|x|^{\alpha_0}\Big)\,\frac{\d z}{|z|^{d+\alpha}}  \\
&\qquad+ \e^{V(x)}\bigg(\sup_{|z|\ge |x|}\e^{-V(z)}\bigg)\bigg[c_1\int_{\{|z|> 1, |x+z|> |x|\}}\frac{\d z}{|z|^{d+\alpha}}+\int_{\{|z|> 1, |x+z|> |x|\}}\frac{\d z}{|z|^{d+\alpha-\alpha_0}}\bigg] \\
&\le -\ff{\e^{V(x)}|x|^{\alpha_0}}2 \bigg(\inf_{|z|\le 1}\e^{-V(z)}\bigg)\int_{\{|z|> 1, |x+z|\le 1\}}  \,\frac{\d z}{|z|^{d+\alpha}}\\
&\qquad
+ \ff{\e^{V(x)}}{(1+|x|)^{d+\aa}\Phi(0)} \bigg[ c_1\int_{\{|z|> 1\}}\frac{\d z}{|z|^{d+\alpha}}+\int_{\{|z|> 1\}}\frac{\d z}{|z|^{d+\alpha-\alpha_0}} \bigg] \\
&\le-c_3\frac{\e^{V(x)}}{(1+|x|)^{d+\alpha}}|x|^{\alpha_0}.\end{align*}
On the other hand, using again the facts that $2\<x,
z\>=|x+z|^2-|x|^2-|z|^2$ for all $x$, $z\in\R^d$, and $b^{\alpha_0}-
a^{\alpha_0}\le \alpha_0a^{\alpha_0-1}(b-a)$ for any $a$, $b\ge0$,
we see that for $|x|$ large enough,
\begin{align*}& \int_{\{|z|\le 1\}}\Big(\phi(x+z)-\phi(x)-\<\nabla \phi(x), z\>\Big)\,\frac{ \e^{V(x)-V(x+z)}}{|z|^{d+\alpha}}\,\d z\\
&\le \alpha_0|x|^{\alpha_0-1}\int_{\{|z|\le 1\}}\Big(|x+z|-|x|-\frac{\<x, z\>}{|x|}\Big)\,\frac{  \e^{V(x)-V(x+z)}}{|z|^{d+\alpha}}\,\d z\\
&=\frac{1}{2}\alpha_0|x|^{\alpha_0-2} \int_{\{|z|\le 1\}}\Big(|z|^2-(|x|-|x+z|)^2\Big)\,\frac{  \e^{V(x)-V(x+z)}}{|z|^{d+\alpha}}\,\d z\\
&\le\frac{1}{2}\alpha_0\bigg(\sup_{|z|\le1}\e^{-V(x+z)}\bigg)|x|^{\alpha_0-2}\e^{V(x)}\int_{\{|z|\le 1\}}\,\frac{\d z}{|z|^{d+\alpha-2}} \\
&=\Bigg[\frac{1}{2}\alpha_0\bigg(\int_{\{|z|\le 1\}}\,\frac{\d z}{|z|^{d+\alpha-2}} \bigg)\bigg(\sup_{|z|\ge |x|-1}\e^{-V(z)}\bigg)|x|^{d+\alpha-2}\Bigg]\times \frac{\e^{V(x)}|x|^{\alpha_0}}{|x|^{d+\alpha}}\\
&\le\frac{c_3}{2}\frac{\e^{V(x)}}{(1+|x|)^{d+\alpha}}|x|^{\alpha_0}, \end{align*} also thanks to  \eqref{A1}. Therefore,
\beq\label{AA2}  \int\Big(\phi(x+z)-\phi(x)-\<\nabla \phi(x), z\>\I_{\{|z|\le 1\}}\Big)
\frac{\e^{V(x)-V(x+z)}}{|z|^{d+\alpha}}\,\d z  \le-\frac{c_3}{3}\frac{\e^{V(x)}}{(1+|x|)^{d+\alpha}}\phi(x)\end{equation} holds for large enough $|x|$.

Finally, according to \eqref{A1}, we find that for $|x|$ large enough
\begin{align*} &\int_{\{|z|\le 1\}}
\<\nn \phi(x), z\>\Big(\e^{V(x)-V(x+z)}-1\Big)\frac{\d z}{|z|^{d+\alpha}}\\
&=\bigg|\e^{V(x)}\int_{\{|z|\le 1\}} \<\nn\phi(x),z\>
\Big(\e^{-V(x+z)}-\e^{-V(x)}\Big)\frac{\d z}{|z|^{d+\alpha}}
 \bigg|\\
&\le \alpha_0 \bigg(\int_{\{|z|\le
1\}} \frac{\d z}{|z|^{d+\alpha-2}}\bigg)\e^{V(x)}|x|^{\alpha_0-1}\Big(\sup_{|z|\ge |x|-1} |\nn e^{-V}(z)|\Big)\\
& \le \frac{c_3}{6}\frac{\e^{V(x)}}{(1+|x|)^{d+\alpha}}\phi(x).
\end{align*}
Combining this with (\ref{AA1}) and (\ref{AA2}), and using the expression of $L_{\aa,V}$ in Proposition \ref{P2.2}, we conclude that
$$ L_{\aa, V}\phi(x)\le -\frac{c_3}{8}\frac{\e^{V(x)}}{(1+|x|)^{d+\alpha}}\phi(x)\le -\ff{c_3}8 \Phi(|x|) \phi(x)$$ holds for large enough $|x|$.
\end{proof}

\section{Proof of Theorem \ref{T1.1}}

In the spirit of \cite[Theorem 2.10]{CGWW}, to derive functional
inequalities using the Lyapunov condition confirmed in Proposition
\ref{P2.3}, we need only to verify  the corresponding local
inequality. So, we first present two lemmas concerning the local
super Poincar\'e inequality and the local Poincar\'e inequality.

\begin{lem} \label{L3.1}
There exists a constant $c>0$ such that for any $s$, $r>0$ and
any $f\in C_0^\infty(\R^d)$,
\beg{equation*}\beg{split} \int_{B(0,r)}f(x)^2 \e^{-V(x)}\d x \le &s\iint_{B(0,r)\times B(0,r)}\frac{(f(y)-f(x))^2}{|y-x|^{d+\alpha}}\,\d y \e^{-V(x)}\,\d x \\
 &+ \ff{cH(r)^{2+d/\aa}}{h(r)^{1+d/\aa}} \big(1+s^{-d/\aa}\big)\bigg(\int_{B(0,r)}|f|(x)\e^{-V(x)}\d x\bigg)^2.\end{split}\end{equation*}
  \end{lem}

\begin{proof}
 Note that the Sobolev inequality of dimension $2d/\aa$ for fractional Laplacians
holds uniformly on balls, e.g.\ see \cite[Section 2]{CK}.
 Then, according to  \cite[Corollary 3.3.4]{WBook} (see also \cite[Theorem 4.5]{W00b}),
there exists a constant $c_1>0$ such that
$$
\int_{B(0,r)} f^2(x)\,\d x\le s\iint_{B(0,r)\times B(0,r)}
\frac{(f(y)-f(x))^2}{|y-x|^{d+\alpha}}\,\d y\,\d x+
c_1\big(1+s^{-d/\aa}\big)  \bigg(\int_{B(0,r)} |f(x)|\,\d
x\bigg)^2$$ holds for all $f\in C_0^\infty(\R^d)$ and all $s,r>0.$
Therefore, for any $r,s>0,$
\begin{align*} \int_{B(0,r)}f^2(x)\e^{-V(x)}\,\d x
 &\le \ff 1 {h(r)}\int_{B(0,r)}f^2(x)\,\d x\\
&\le \ff 1 {h(r)}\bigg\{s\iint_{B(0,r)\times B(0,r)} \frac{(f(y)-f(x))^2}{|y-x|^{d+\alpha}}\,\d y\,\d x\\
&\qquad\qquad+c_1\big(1+s^{-d/\aa}\big)
 \bigg(\int_{B(0,r)}|f(x)|\,\d x\bigg)^2\bigg\}\\
&\le \ff{sH(r)}{h(r)}\iint_{B(0,r)\times B(0,r)} \frac{(f(y)-f(x))^2}{|y-x|^{d+\alpha}}\,\d y\e^{-V(x)}\,\d x\\
&\quad+\ff{c_1(1+s^{-d/\aa})H^2(r)}{h(r)} \bigg(\int_{B(0,r)} |f(x)|\e^{-V(x)}\,\d x\bigg)^2. \end{align*}
This implies the  desired assertion   by replacing $s$ with $sh(r)H(r)^{-1}$.
\end{proof}

\beg{lem}\label{L3.2}   For any $r>0$ and $f\in C_0^\infty(\R^d)$,
\beq\label{LLP} \mu_V(f^21_{B(0,r)}) \le  \Psi_2(r) \E_{\aa,V} (f,f)
+\ff{\mu_V(f1_{B(0,r)})^2}{\mu_V(B(0,r))}.\end{equation} 
Consequently,
the weak Poincar\'e inequality $(\ref{WP})$ holds for
\beq\label{*N}\tt\bb(r):= \inf \Big\{  \Psi_2(R):\
\mu_V(B(0,R)^c)\le \ff r {1+r}\Big\}<\infty,\ \ r>0.\end{equation}
\end{lem}

\beg{proof} By the Cauchy-Schwarz inequality,
\beg{equation*} \beg{split} &\int_{B(0,r)}\bigg(f(x)-\frac{1}{\mu_V(B(0,r))}\int_{B(0,r)}f(x)\,\mu_V(\d x)\bigg)^2\mu_V(\d x)\\
&=\int_{B(0,r)}\bigg(\frac{1}{\mu_V(B(0,r))}\int_{B(0,r)}(f(x)-f(y))\,\mu_V(\d y)\bigg)^2\mu_V(\d x)\\
&\le \frac{1}{\mu_V(B(0,r))^2}\int_{B(0,r)}\bigg(\int_{B(0,r)}(f(x)-f(y))^2\frac{e^{V(y)}}{|y-x|^{d+\alpha}}\,\mu_V(\d y)\bigg)\\
&\qquad\qquad\qquad\qquad\qquad\times\bigg(\int_{B(0,r)}\frac{|y-x|^{d+\alpha}}{e^{V(y)}}\mu_V(\d y)\bigg)\,\mu_V(\d x)\\
&\le\Psi_2(r) \iint_{B(0,r)\times B(0,r)}
\frac{(f(x)-f(y))^2}{|x-y|^{d+\alpha}}\,\d y \mu_V(\d
x).\end{split}\end{equation*} So, the inequality (\ref{LLP}) holds,
which implies the desired weak Poincar\'e inequality according to
\cite[Theorem 3.1]{RW01} or \cite[Theorem 4.3.1]{WBook}.
 \end{proof}

Since $\mu_V(\R^d)<\infty$, most likely we have $\int_{\R^d}
\e^{-2V(y)}\d y<\infty$, so that $\Psi_2(R)\le c_1R^{d+\aa}$ holds
for some constant $c_1>0$ and all $R\ge 1.$ In this case  there
exists a constant $c>0$ such that $\tt\bb$ in (\ref{*N}) satisfies
$$\tt\bb(r)\le c+c\,\inf\Big\{R^{d+\aa}:\ \mu_V(B(0,R)^c)\le \ff r {1+r}\Big\}<\infty,\ \ r>0.$$
In many cases this $\tt\bb$ is however not sharp, for instance, in
the proofs of Corollaries \ref{C1.2} -\ref{C1.4} we will use $\Psi_1$
rather than $\Psi_2$ in
  Theorem \ref{T1.1}(3) to derive sharp estimates on $\tt\bb$.

\beg{proof}[Proof of Theorem \ref{T1.1}] 
First, according to Proposition \ref{P2.3}, we have
 $$
1_{B(0,r)^c}\le  \frac{1}{C_1\Phi(r)}\frac{-L_{\aa,V}\phi}{\phi}+\frac{C_2}{C_1\Phi(r)}\I_{B(0,r_0)},\ \ \  r\ge r_0.
$$  Then, for any $f\in C_0^\infty(\R^d)$,
\beq\label{AF}\mu_V(f^2\I_{B(0,r)^c})\le \frac{1}{C_1\Phi(r)}\mu_V\Big(f^2\frac{-L_{\aa,V}\phi}{\phi}\Big)+
 \frac{C_2}{C_1\Phi(r)}\mu_V(f^2\I_{B(0,r_0)}).\end{equation}
By Proposition \ref{P2.2} and the fact that
\beg{equation*}\beg{split} \Big(\ff{f^2(x)}{\phi(x)}- \ff{f^2(y)}{\phi(y)}\Big)(\phi(x)-\phi(y)) &=f^2(x)+f^2(y) -\Big(\ff{\phi(y)}{\phi(x)}f^2(x)+\ff{\phi(x)}{\phi(y)}f^2(y)\Big)\\
&\le f^2(x)+f^2(y) -2 |f(x)f(y)|\\
&\le (f(x)-f(y))^2,\end{split}\end{equation*} we obtain
\begin{equation}\label{proof}\mu_V\big(f^2\frac{-L_{\aa,V}\phi}{\phi}\big)\le
\E_{\aa,V}(f,f).\end{equation} Therefore, (\ref{AF}) implies
\beq\label{JJ1} \mu_V(f^2\I_{B(0,r)^c})\le
\frac{1}{C_1\Phi(r)}\E_{\aa,V}(f,f)+
 \frac{C_2}{C_1\Phi(r)}\mu_V(f^2\I_{B(0,r_0)}),\ \ r\ge r_0.\end{equation}
We are now to prove (1) and (2) in Theorem \ref{T1.1} respectively.

(1) According to \cite[Theorem 4.5 and Theorem 3.2]{W00b}, the local super Poincar\'e inequality in Lemma \ref{L3.1} implies that the associated Markov semigroup on $B(0,r)$
 has a uniformly bounded density, and hence the spectrum of the associated generator is discrete. Moreover, it is easy to see that the Dirichlet form on $B(0,r)$ is irreducible so that $0$ is a simple eigenvalue of the generator, we conclude that the spectral gap exists. Equivalently, for any $r>0$ there exists a constant $C(r)>0$ such that the local Poincar\'e inequality
 \beq\label{LP}\mu_V( f^21_{B(0,r)} ) \le C(r) \iint_{B(0,r)\times B(0,r)} \frac{(f(y)-f(x))^2}{|y-x|^{d+\alpha}}\,\d y \mu_V(\d x)\end{equation} holds for all $f\in C_0^\infty(\R^d)$ with
 $\mu_V(f1_{B(0,r)})=0.$ This, together with (\ref{JJ1}) implies the defective Poincar\'e inequality
 $$\mu_V(f^2) \le c_1 \E_{\aa, V}(f,f) + c_2 \mu_V(|f|)^2$$ for some constants $c_1,c_2>0$; and due to \cite[Theorem 3.1]{RW01}, \eqref{LP} also implies the weak Poincar\'e inequality of $\E_{\aa, V}$. According to \cite[Proposition 1.3]{RW01}, these two inequalities then imply the desired  Poincar\'e inequality.

(2) Now, assume that $\Phi(r)\uparrow\infty$ as $r\uparrow\infty.$  By Lemma \ref{L3.1}, there exists a constant $c>0$ such that
$$\mu_V(f^2\I_{B(0,r)})\le s \E_{\aa,V}(f,f)+\bb(r,s)\mu_V(|f|)^2,\quad s,r>0, f\in C_0^\infty(\R^d)$$ holds for
$$\bb(r,s):= \ff{cH(r)^{2+d/\aa}}{h(r)^{1+d/\aa}} \big(1+s^{-d/\aa}\big).$$ Combining this with (\ref{JJ1}) and (\ref{LP}) with $r=r_0$, we may find a constant $c_0>0$ such that, for any $r\ge r_0,$
\beg{equation*}\beg{split} \mu_V(f^2) &= \mu_V(f^21_{B(0,r)}) +\mu_V(f^21_{B(0,r)^c})\\
&\le \Big(s+\ff{c_0}{\Phi(r)}\Big) \E_{\aa, V} (f,f)+ \big(c_0
+\bb(r, s)\big)\mu_V(|f|)^2,\ \ s>0, f\in
C_0^\infty(\R^d).\end{split} \end{equation*} Letting
$s_0=c_0/\Phi(r_0)$  and taking $r=\Phi^{-1}(c_0/s)$, which is
larger than $r_0$ if $s\in (0,s_0),$ we obtain
$$\mu_V(f^2) \le 2 s \E_{\aa,V}(f,f) +\big\{c_0+ \bb(\Phi^{-1}(c_0/s), s)\big\} \mu_V(|f|)^2,\ \ s\in (0,s_0), f\in C_0^\infty(\R^d).$$ Replacing $s$ by $s/2$, we $$ \mu_V(f^2) \le  s \E_{\aa,V}(f,f) +\big\{c_0+ \bb(\Phi^{-1}(2c_0/s), s/2)\big\} \mu_V(|f|)^2,\ \ s\in (0,2s_0), f\in C_0^\infty(\R^d).$$
Noting that
$$\bb(\Phi^{-1}(2c_0/s), s/2)=\ff{c \big\{H\circ\Phi^{-1}(2c_0/s)\big\}^{2+d/\aa}}{\{h\circ \Phi^{-1}(2c_0/s)\big\}^{1+d/\aa}}\big(1+2^{d/\aa}s^{-d/\aa}\big),$$
this implies the  super Poincar\'e inequality with the desired $\bb$ for some constants $c_1,c_2>0$ and all $s\in (0, 2s_0).$ Then the inequality holds also
for $s\ge 2s_0$ with a possibly large constant $c_1$ by taking
$\bb(s)=\bb(2s_0)$ for $s\ge 2s_0$.

(3) Let $V_0 (x)=\ff {d+\aa} 2\log (1+|x|^2).$  Then Theorem \ref{T1.1}(1) implies that
the Poincar\'e inequality
\beq\label{PP1} \mu_{V_0} (f^2) \le C \mu_{V_0}(\GG(f,f)),\ \ f\in C_0^\infty(\R^d), \mu_{V_0}(f)=0\end{equation}  holds for some constant $C>0$, where
$$\GG(f,f)(x):= \int_{\R^d} \ff{|f(y)-f(x)|^2}{|x-y|^{d+\aa}}\, \d y,\ \ x\in\R^d.$$
For any $R>0$ and any $f\in C_0^\infty(\R^d)$, it follows from (\ref{PP1}) that
\begin{align*}&\int_{B(0,R)}\bigg(f(x)-\frac{1}{\mu_V(B(0,R))}\int_{B(0,R)}f(x)\,\mu_V(\d x)\bigg)^2\,\mu_V(\d x)\\
&=\inf_{a\in\R} \int_{B(0,R)}\bigg(f(x)-a\bigg)^2\,\e^{-V(x)}\d x\\
& \le \int_{B(0,R)}\bigg(f(x)-\mu_{V_0}(f)\bigg)^2\,\e^{-V(x)}\d x\\
&\le \bigg(\sup_{|x|\le R}\frac{(1+|x|^2)^{(d+\alpha)/2}}{e^{V(x)}}\bigg)\int_{B(0,R)}\bigg(f(x)-\mu_{V_0}(f)\bigg)^2\,\e^{-V_0(x)}\d x\\
&\le C\bigg(\sup_{|x|\le R}\frac{(1+|x|^2)^{(d+\alpha)/2}}{e^{V(x)}}\bigg)\int_{\R^d}\GG(f,f) (x)\e^{-V_0(x)}\d x  \\
&\le c\Psi_1(R)\int_{\R^d} \GG(f,f)(x)\e^{-V(x)}\d x.\end{align*} That
is,
$$\mu_V(f^21_{B(0,R)})\le c\Psi_1(R) \E_{\aa,V}(f, f)+\frac{\mu_V(f1_{B(0,R)})^2}{\mu_V(B(0,R))}.$$ Combining this with Lemma \ref{L3.2} we obtain
$$\mu_V(f^21_{B(0,R)}) \le  \big\{(c\Psi_1(R)\big)\land \Psi_2(R)\big\} \E_{\aa,V} (f,f) +\ff{\mu_V(f1_{B(0,R)})^2}{\mu_V(B(0,R))}.$$
The required weak Poincar\'e inequality then  follows from
\cite[Theorem 3.1]{RW01} or \cite[Theorem 4.3.1]{WBook}.
  \end{proof}

\begin{rem} The formula \eqref{proof} for diffusion operators is easily derived by using a chain rule,
e.g.\ see \cite[(2.2)]{CGWW}; and  the proof of it for symmetric
jump processes is based on the large derivation, see \cite[Lemma
2.12]{CGWW}. Our proof here is more straightforward.
\end{rem}
\section{Proofs of Corollaries}

In all these Corollaries, the sufficiency for the Poincar\'e/super Poincar\'e/weak Poincar\'{e} inequalities will be confirmed by Theorem \ref{T1.1}. To verify the necessary, we will make use of the   reference functions $g_n\in C^\infty(\R^d),  n\ge 1,$   such that $|\nn g_n|\le 2/n$ and
$$g_n(x)\beg{cases} =0, &\text{if}\ |x|\le n, \\
\in [0,1], &\text{if} \ |x|\in [n, 2n],\\
=1, &\text{if}\ |x|\ge 2n.\end{cases}$$
Then there exists a constant $c>0$ independent of $n$ such that
\beq\label{NN}\beg{split}  \GG(g_n,g_n)(x)&:=\int_{\R^d} \ff{|g_n(y)-g_n(x)|^2}{|x-y|^{d+\aa}}\, \d y \\
&\le \ff 4{n^2}\int_{|x-y|\le n} \ff 1 {|y-x|^{d+\aa-2}}\,\d y+\int_{|x-y|\ge n} \ff 1 {|x-y|^{d+\aa}}\,\d y\\
&\le \ff c {n^\aa},\ \ n\ge 1.\end{split}\end{equation}

\beg{proof}[Proof of Corollary \ref{C1.2}] Obviously, for any $\vv>0$, the function
$$V(x):= \ff {d+\vv} 2 \log (1+|x|^2),\ \ x\in\R^d$$ satisfies condition (\ref{A1}).

(1) If $\vv\ge \aa$, we have $\Phi(0)>0$, so that the Poincar\'e inequality follows from Theorem \ref{T1.1}(1). To disprove the Poincar\'e inequality
for $\vv\in (0,\aa)$, let us take the reference function $g_n$ introduced above. Obviously,
$$\mu_V(g_n)^2 \ge \ff{c_1}{n^\vv},\ \ \mu_V(g_n)^2\le \ff{c_2}{n^{2\vv}},\ \ n\ge 1$$ hold for some constants $c_1,c_2>0$. Combining this with (\ref{NN}) we see that
$$\lim_{n\to\infty} \ff{\E_{\aa,V}(g_n,g_n)}{\mu_V(g_n)^2-\mu_V(g_n)^2} \le \lim_{n\to\infty} \ff{cn^{-\alpha}}{c_1 n^{-\vv}- c_2 n^{-2\vv}}=0$$
provided $\vv\in (0,\aa).$ Thus, for any constant $C>0$, the Poincar\'e inequality (\ref{P}) does not hold.

(2) We first prove that if $\vv\le\aa$, then for any $\bb: (0,\infty)\to (0,\infty)$ the super Poincar\'e inequality (\ref{SP}) does not hold. Indeed, if this inequality holds, then
$$\ff{c_1}{n^\vv}\le \ff{cr}{n^\aa} + \ff{c_2\bb(r)}{n^{2\vv}},\ \ r>0, n\ge 1$$ holds for some constants $c,c_1,c_2>0$.
Since $\vv \in (0, \aa]$, we obtain
$$c_1\le \lim_{n\to\infty} \ff{cr}{n^{\aa-\vv}} +\lim_{n\to\infty} \ff{c_2\bb(r)}{n^\vv} \le cr,\ \ r>0.$$ Letting $r\to 0$ we conclude that $c_1\le 0$, which is however impossible.

Next, let $\vv>\aa$, we aim to confirm the super Poincar\'e inequality with the desired function $\bb(r)$.  It is easy to see that
$$h(r)=1,\ \ H(r)= (1+r^2)^{(d+\vv)/2}, \ \ r>0$$ and $\Phi(r)\ge   c_3r^{\vv-\aa}$ for $r$ large so that
$$\Phi^{-1}(c_2r^{-1})\le   c_4 r^{-1/(\vv-\aa)}\ \  \textrm{ for } r>0 \textrm{ small} . $$
Hence, the function $\bb$ given in Theorem \ref{T1.1} (2) satisfies
$$\bb(r)\le  c\Big(1+ r^{-\ff d\aa -\ff{(d+\vv)(2\aa+d)}{\aa(\vv-\aa)}}\Big),\ \ r>0$$ for some constant $c>0.$  The equivalence of the concrete super Poincar\'e inequality and the correspondinf bound of $\|P_t^{\aa,V}\|_{L^1(\mu_V)\to L^2(\mu_V)}$ then follows from
\cite[Theorem 4.5(2)]{W00b} (see also \cite[Theorem
3.3.15(2)]{WBook}).

(3) It is easy to see that $\Psi_1(R) ={\rm O} (R^{\aa-\vv})$ for
large $R$. Then the desired weak Poincar\'e inequality follows from
Theorem \ref{T1.1}(3). According to \cite[Corollary 2.4(2)]{RW01}
(see also \cite[Theorem 4.1.5(2)]{WBook}), we have the claimed
bound of $\|P_t^{\aa,V}-\mu_V\|_{L^\infty(\mu_V)\to L^2(\mu_V)}^2$.
On the other hand, for $g_n$ presented in the beginning of this
section, we have $\|g_n\|_\infty\le 1,$ $\mu_V(g_n^2)- \mu_V(g_n)^2
\ge c_1 n^{-\vv}$ for some constant $c_1>0$, and due to (\ref{NN})$, \E_{\aa,V}(g_n,g_n)\le {c} {n^{-\aa}}.$ Then (\ref{WP}) implies
that
$$\ff{c}{n^\aa}\tt\bb(r)\ge \ff{c_1}{n^\vv}  -r,\ \ r>0.$$ Taking $r_n= \ff {c_1} {2n^\vv}$  which goes to zero as $n\to\infty,$ we obtain
$$\liminf_{n\to\infty} r_n^{(\aa-\vv)/\vv} \tt\bb(r_n)>0.$$ Thus, (\ref{WP}) does not hold if $\lim_{r\to 0} r^{(\aa-\vv)/\vv}\tt\bb(r)=0.$
\end{proof}

\beg{proof}[Proof of Corollary \ref{C1.3}] Since when $\vv>0$ we
have  $\Phi(0)>0$, the Poincar\'e inequality holds due to Theorem
\ref{T1.1}(1). According to e.g.\ \cite[Corollary 1.3(1)]{W00b}, the
super Poincar\'e inequality with $\bb(r)=\exp\big(c(1+r^{-1})\big)$
for some constant $c>0$ is equivalent to the log-Sobolev inequality
(\ref{LS}) for some constant $C>0,$ we conclude that (1) and (2)
imply (3). So, it suffices to prove (1), (2) and (4).

(1) As in the proof of Corollary \ref{C1.2}(2), when $\vv\le 0$ the super Poincar\'e inequality does not hold. Let $\vv>0$. We have
$$\e^{-V(x)}= \ff 1 {(1+|x|^2)^{(d+\aa)/2} \log^\vv(\e+|x|^2)},\ \ x\in \R^d.$$
Then it is easy to see that
$$h(r)=1,\ \ H(r)= {(1+r^2)^{(d+\aa)/2}\log^\vv (\e+r^2)}$$ and $$\Phi(r) \ge {c_0} {\log^\vv (1+r^2)},\ \ r>0$$ holds for some constant $c_0>0.$ So, there exists a constant $c_3>0$ such that
$$\Phi^{-1}(c_2r^{-1})\le \exp[c_3 r^{-1/\vv}],\ \ r>0,$$
and hence, the function $\bb$ given in Theorem \ref{T1.1}(2) satisfies
$$\bb(r)\le \exp[c(1+r^{-1/\vv})],\ \ r>0$$ for some constant $c>0.$
When $\vv>1$, the equivalence of the concrete super Poincar\'e
inequality and the corresponding bound of
$\|P_t^{\aa,V}\|_{L^1(\mu_V)\to L^2(\mu_V)}$ then follows from
\cite[Theorem 4.5(1)]{W00b} (see also \cite[Theorem
3.3.15(1)]{WBook}).

(2) It is easy to see that
$$\mu_V(g_n^2)\ge \ff{c_1} {n^{\aa}\log^\vv (\e +n)},\ \ \mu_V(|g_n|)^2\le \ff{c_2}{n^{2\aa}\log^{2\vv}(\e +n)},\ \ n\ge 1$$ hold for some constants $c_1,c_2>0$. Combining this with (\ref{NN}) and (\ref{SP}), we obtain
$$\ff{c_1}{\log^\vv(\e+n)} \le cr+ \ff{c_2\bb(r)}{n^\aa\log^{2\vv}(\e+n)},\ \ r>0.$$ Taking $r_n=  \ff{c_1}{2c}\log^{-\vv}(\e+n),$ we derive
$$\bb(r_n)\ge \ff{c_1}2  n^\aa\log^\vv(\e+n),\ \ n\ge 1.$$ Therefore,
$$\liminf_{n\to \infty} r_n^{1/\vv}\log \bb(r_n) \ge \aa>0.$$  Thus, the proof of (2) is done.

(4) Let $\vv<0$. Then  there exist constants $C,c>0$ such that
$$\Psi_1(R) \le C \log^{-\vv} (\e +R),\ \ \ \mu_V(B(0,R)^c)\le c R^{-\aa} \log^\vv(\e+R),\ \ R>0.$$
So, the desired weak Poincar\'e inequality follows from Theorem
\ref{T1.1}(3), and the corresponding convergence rate of
$\|P_t^{\aa,V}-\mu_V\|_{L^\infty(\mu_V)\to L^2(\mu_V)}$ follows from
\cite[Corollary 2.4(1)]{RW01}. Finally, the sharpness of $\tt\bb$
can be easily verified using reference functions $g_n, n\ge 1$.
\end{proof}

\beg{proof}[Proof of Corollary \ref{C1.4}]
There exist constants
$C,c>0$ such that
$$\Psi_1(R) \le \frac{C R^\alpha }{\log^\varepsilon(e+R)},\ \ \
 \mu_V(B(0,R)^c)\le c\Big(\log(e+R)\Big)^{-(\varepsilon-1)},\ \ R>0.$$ So, the desired weak Poincar\'e inequality follows from Theorem \ref{T1.1}(3) and the corresponding convergence rate of $\|P_t^{\aa,V}-\mu_V\|_{L^\infty(\mu_V)\to L^2(\mu_V)}$ follows from \cite[Corollary 2.4(3)]{RW01}.
Similar to the part (4) in the proof of Corollary \ref{C1.3}, the
sharpness of $\tt\bb$ can be easily verified using reference
functions $g_n, n\ge 1$.\end{proof}

\beg{proof}[Proof of Corollary \ref{C1.5}] For  the  super
Poincar\'e inequality with desired $\bb$, we need to  prove for
small $r>0$, since     we may always take $\bb$ to be deceasing in
the super Poincar\'e inequality.

(1) Since $$\e^{V(x)}= \exp\big[\log^{1+\vv}(1+|x|^2)\big],$$ it is easy to see that
$h(r)=1,\  H(r)=\exp\big[\log^{1+\vv}(1+r^2)\big]$ and $$\Phi(r)=\ff{H(r)}{(1+r)^{d+\aa}}\ge \exp\Big[\ff 1 2  \log^{1+\vv}(1+r)\Big],\ \ r\ge r_0$$ holds for some constant $r_0>0$. So,
$$H\circ\Phi^{-1}(c_2r^{-1}) =\big\{\Phi\circ\Phi^{-1}(c_2r^{-1})\big\}\cdot\big\{1+\Phi^{-1}(c_2r^{-1})\big\}^{d+\aa}\le cr^{-1}\exp\big[c\log^{1/(1+\vv)} r^{-1}\big]$$ holds for some constant $c>0$ and small $r>0.$
Then (\ref{SP})  with the desired $\bb$ for small $r>0$ follows from Theorem \ref{T1.1}(2), and the corresponding bound of $\|P_t^{\aa,V}\|_{L^1(\mu_V)\to L^\infty(\mu_V)}$ then follows from e.g.\ \cite[Theorem 4.4]{W00b}.

(2) Since $\e^{V(x)}= \exp[(1+|x|^2)^\vv]$, it is easy to see that
$h(r)=1,\  H(r)= \exp\big[ (1+r^2)^\vv\big]$ and  $$\Phi(r)= \ff{\exp [  (1+r^2)^\vv ]}{1+r^{d+\aa}},\ \ r\ge r_0$$ holds for some constant $r_0>0$.
Then there exists a constant $c>0$ such that for small enough $r>0$,
\beg{equation*}\beg{split} H\circ\Phi^{-1}(c_2 r^{-1}) &=\Big\{\Phi\circ \Phi^{-1}(c_2 r^{-1})\Big\} \cdot \Big\{1+\Phi^{-1} (c_2r^{-1}) \Big\}^{d+\aa} \\
&=  c_2  r^{-1} \Big(\Phi^{-1} (c_2r^{-1})^{d+\aa}\Big)\\
&\le c r^{-1}  \log^{(d+\aa)/(2\vv)} (1+
r^{-1}).\end{split}\end{equation*} Therefore, the super Poincar\'e
inequality with the desired $\bb(r)$ for small enough $r>0$ follows
from Theorem \ref{T1.1}(2), and the corresponding bound of
$\|P_t^{\aa,V}\|_{L^1(\mu_V)\to L^\infty(\mu_V)}$ then follows from
\cite[Theorem 4.4]{W00b}.
\end{proof}

\section{Super Poincar\'e inequalities implied by (\ref{EF})}

This section aims to establish the super Poincar\'e inequality using condition (\ref{EF}), so that the assertion in \cite{MRS} for the Poincar\'e inequality is strengthened.  As already indicated in Section 1 that the resulting super Poincar\'e inequality is normally worse than that presented in Theorem \ref{T1.1}.

For fixed $V\in C^2(\R^d)$ such that $\mu_V$ is a probability measure, let $h, H$   be as in Theorem \ref{T1.1}, and let
$$ W_\dd(r)= \inf_{|x|\ge r}\bigg(\dd |\nabla V(x)|^2  -\Delta V(x)\bigg),\ \ r>0. $$

\begin{thm}\label{TA} Let $V\ge 0$.
If there exists a constant $\dd\in(0,1)$ such that
$$
\lim_{|x|\to\infty} \big\{ \dd|\nabla V|^2 -\Delta V\big\}=\infty,$$
Then there exist constants $c_1,c_2>0$ such that the super Poincar\'{e} inequality $(\ref{SP})$ holds for
$$ \beta(r)=c_1\bigg(1+r^{-d/\alpha}H^{2+d/2}\Big(W_\dd^{-1}
\big(c_2r^{-2/\alpha}\big)\Big)h^{-(1+d/2)}\Big(W_\dd^{-1}
\big(c_2r^{-2/\alpha}\big)\Big)\bigg),\ \ r>0.
$$ In particular, if $V(x)= (1+|x|^2)^\vv$ for $\vv>\ff 1 2$, then there exists a constant $c>0$ such that the super Poincar\'e inequality holds for
$$\bb(r) = \exp\Big(c(1+r^{-2\vv/(\alpha(2\vv-1))})\Big),\ \ r>0.$$
 \end{thm}

\begin{proof} We only prove the first assertion, since the second one is a simple consequence. Let
$$
L_Vf=\Delta f-\<\nabla V, \nabla f\>,\ \ f\in C^2(\R^d).
$$ Then
 $$
\E_V(f,g):=\int_{\R^d}\<\nabla f, \nn g\>
 \,\d\mu_V=-\int_{\R^d}f L_V g\, \d\mu_V,\ \ f,g\in C_0^\infty(\R^d).
$$Hence, the Friedrichs extension $(L_V,\D(L_V))$ of $(L_V, C_0^\infty(\R^d))$  in $L^2(\mu_V)$ is a negatively definite self-adjoint operator. Let $(-L_V)^{\aa/2}$ be the associated fractional operator.
Let   $\varphi=\e^{(1-\dd) V}$. We have
 $$
\frac{L_V \varphi}{\varphi}=-(1-\dd)\Big( \dd |\nabla V|^2
-\Delta V\Big).
$$
Then, by the assumption on $V$ and \cite[Theorem 2.10]{CGWW}, there exist constants $c_3,c_4>0$ such that
the super Poincar\'{e} inequality
$$
\mu_V(f^2)\le r \E_V(f,f)+\beta_V(r)\mu(|f|)^2,\quad r>0, f\in C_0^\infty(\R^d), \mu(f)=0 $$  holds for
$$
\beta_V(r):=c_3\bigg(1+r^{-d/2}H^{2+d/2}\Big(W^{-1}
\big(c_4r^{-1}\big)\Big)h^{-(1+d/2)}\Big(W^{-1}
\big(c_4r^{-1}\big)\Big)\bigg),\ \ r>0.
$$
According to \cite[Corollary 2.1]{W07} or the proof
of \cite[Proposition 9]{SW11}, this implies
\beq\label{W*} \mu_V(f^2)\le r\int_{\R^d} f(-L_V)^{\aa/2}f\, \d\mu_V   +\frac{8}{\alpha}\beta_V\!\Big((r/{4})^{2/\alpha}\!\Big)\mu_V(|f|)^2,\ \ r>0\end{equation}    for all $f\in C_0^\infty(\R^d)$ with $ \mu_V(f)=0.$
 A close inspection of the arguments in \cite[Section 3]{MRS}
(see \cite[Lemma 3.2 and Lemma 3.3]{RS} for details)
yields that there is a  constant $C>0$ such that for all $f\in
C_0^\infty(\R^d)$ with $\mu_V(f)=0,$
$$
 \int_{\R^d} f (-L_V)^{\alpha/2}f \,\d\mu_V \le C \iint_{\R^d\times\R^d}
\frac{(f(y)-f(x))^2}{|y-x|^{d+\alpha}}\,\d y\,\mu_V(\d x)=C \E_{\aa, V}(f,f).
$$Combining this with  (\ref{W*}), we obtain
$$
\mu_V(f^2)\le r \E_{\aa,V}(f,f)
+\frac{8}{\alpha}\beta_V\!\Big(\Big(\frac{r}{4C}\Big)^{2/\alpha}\!\Big)
\mu_V(|f|)^2,\quad r>0, f\in C_0^\infty(\R^d), \mu_V(f)=0.
$$
  Then the
desired assertion follows immediately.
\end{proof}

Similarly, combining the proof above with \cite[Theorem 3.1]{RW01}
and \cite[Corollary 2.2]{W07} (or the proof of \cite[Proposition
9]{SW11}), we have the following result for weak Poincar\'{e}
inequalities for stable-like Dirichlet forms, which is normally less
sharp than that given in Theorem \ref{T1.1}.

\begin{thm}
For any $V\in C^2(\R^d)$ such that $\mu_V$ is a probability measure, there exist constants $c_1,c_2>0$ such that the
weak Poincar\'{e} inequality \eqref{WP} holds for
$$
\widetilde{\beta}(r)=c_1\,U({c_2r^{\alpha/2}})^2\exp\Big(2\delta_{U({c_2r^{\alpha/2}})}(V)\Big),\ \ r>0,
$$ where
$$U(r)=\inf\bigg\{s>0:\int_{|x|>s}e^{-V(x)}\,dx\le r/(1+r)\bigg\}\,\,
\textrm{ and }\,\, \delta_{r}(V)=\sup_{|x|\le r}V(x).
$$
\end{thm}

\paragraph{Acknowledgements.} The authors are indebted to the referee and an associate
editor for their suggestions. The authors also would like to thank
Dr.\ Xin Chen and Professors Ren\'{e} L. Schilling and Renming Song
for helpful comments on earlier versions of the paper.

\end{document}